\documentclass[a4paper]{article}

\usepackage{amssymb}
\usepackage{amsthm}
\usepackage{amsmath}
\usepackage{mathtools}
\usepackage{subfig}
\usepackage{color}
\usepackage{tikz}
\usepackage{pgfplots}
\usetikzlibrary{calc,patterns}
\usetikzlibrary{positioning} 
\usetikzlibrary{calc}
\usepackage{pgfplots}
\pgfplotsset{width=10cm,compat=1.9}

\usepackage{hyperref}
\usepackage[all]{hypcap}

\newsavebox{\tempbox}

\theoremstyle{remark}
\newtheorem{remark}{Remark}
\theoremstyle{plain}
\newtheorem{lemma}{Lemma}
\theoremstyle{plain}
\newtheorem{theorem}{Theorem}
\theoremstyle{plain}

\def \bb{{\bf b}}
\newcommand{\bu}{\bar{u}}
\newcommand{\bt}{\boldsymbol{t}}

\newcommand{\cin}{C_{\mathrm{inv}}}

\title{Analysis of Flux Corrected Transport Schemes for Evolutionary Convection-Diffusion-Reaction Equations}
\author{Abhinav Jha
\footnote{Corresponding Author, RWTH Aachen University, Applied and Computational Mathematics, Schinkelstra\ss e 2, 52062, Aachen, Germany \texttt{jha@acom.rwth-aachen.de}}, Naveed Ahmed\footnote{Gulf University for Science \& Technology,Block 5, Building 1,Mubarak Al-Abdullah Area, West Mishref Kuwait, \texttt{ahmed.n@gust.edu.kw}}}
\date{}

\begin{document}
\maketitle
\begin{abstract}
We report in this paper the analysis for the linear and nonlinear version of the flux corrected transport (FEM-FCT) schemes in combination with the backward Euler time-stepping scheme applied to a time-dependent convection-diffusion-reaction problems. We present the stability and error estimates for the linear and nonlinear FEM-FCT scheme. Numerical results confirm the theoretical predictions.

\textbf{Keywords:}
evolutionary convection-diffusion-reaction equations; finite element method flux corrected transport; finite element error analysis

\textbf{AMS subject classifications:} 65N12, 65N30
\end{abstract}

\section{Introduction}
The time-dependent convection diffusion reaction equations are used to model many physical processes arising in computational fluid dynamics. When convection dominates the diffusion, we have the presence of layers in the interior and the boundary of the domain. 

To overcome the difficulty of instabilities, stabilization schemes are applied. 
One of the most popular methods is the Streamline Upwind Petrov-Galerkin (SUPG) method which was introduced in \cite{HB79, BH81} for the steady state problems.  However, the drawback of the SUPG method for time-dependent problem is the fact that for ensuring the strong consistency the time derivative, the second order derivative and the source term have to be included into the stabilization term. The SUPG method in space combined with different time discretization was investigated in~\cite{Bur10}. In~\cite{JN11}, the SUPG method is combined with the backward Euler and Crank-Nicolson methods. It turns out that the stabilization vanishes if the time step length approaches zero. In the case of time independent convection and reaction coefficients, error estimates are derived which allow the stabilization parameters to be chosen similar to the steady-state case.

Alternative to SUPG are the symmetric stabilization schemes such as the local projection stabilization (LPS)~\cite{AMTX11,BB04, BB06, AM15}, the continuous interior penalty method (CIP)~\cite{BF09}, the subgrid scale modeling (SGS) \cite{Gue99, Lay02}. A comparison of the SUPG method with the LPS method can be found in~\cite{AM15, AJ15}.

Another stabilization approach is the so-called algebraic schemes introduced in \cite{Zal79} and was combined with finite element discretization in \cite{PC86} which works on an algebraic level rather than a variational level. In literature, these schemes are referred to as flux-corrected transport (FEM-FCT). Some of the prominent work in this direction has been done by Kuzmin and co-authors and can be found in \cite{Ku06, Ku07, LKSM17, Loh19}. It has been noted in \cite{JR10} that the FEM-FCT schemes performed better than the SUPG schemes. A comparison for different stabilization schemes and discretization can be found in \cite{ACFFJL10}.
 

This technique's steady-state counterpart is referred to as the algebraic flux correction (AFC) schemes introduced in \cite{Ku07}. The numerical analysis of the AFC schemes has been developed recently in \cite{BJK16} and a modification to the limiter definition is presented in \cite{BJK17}. This paper provided a new insight for the AFC schemes as this was the first finite element error analysis for the schemes. The stabilization parameters are referred to as limiters, and a comprehensive comparison of the results based on different limiters is presented in \cite{BJKR18}. One of the drawbacks of the AFC schemes is their nonlinear nature. A study on the solvers for these schemes can be found in \cite{JJ19, JJ18}. This nonlinearity also arises in the FEM-FCT schemes, and a study for different types of solvers is presented in \cite{JN12}. The FCT schemes have been successfully applicable only to lower-order elements and we consider the same in our analysis.

As mentioned above, the FEM-FCT schemes are nonlinear in nature when combined with a $\theta-$scheme, but one can also linearize the scheme using the solution at the previous time step. A comparison of the linear and the nonlinear schemes are present in \cite{JS08} where it is shown that the linear FEM-FCT has a better ratio of accuracy and efficiency. The solvability of the nonlinear FEM-FCT schemes has been presented recently in \cite{JKK21} and the existence and uniqueness of the solution are shown in \cite{JK21} where the analysis from \cite{BJK16} is extended. This paper presents the finite element error analysis of the FEM-FCT schemes with backward Euler as the time-stepping. To the best of our knowledge, this is the first work in this direction. The positivity of the solution has already been discussed in \cite{Ku09} where the positivity is independent of the choice of limiters'. Here we will be considering both the linear and the nonlinear version of the FEM-FCT schemes mentioned in \cite{Ku09}. We assume the limiter's general properties in order to obtain optimal order of convergence on the time and space in the $L^2$ and the $H^1$ norm. The analysis is performed in the system's natural norm, which we refer to as the FCT norm. 

The work's main findings are the CFL-like stability condition, the optimal convergence rate for the $L^2$ and the $H^1$ norm, and the sub-optimal convergence rate for the FCT norm. Numerical simulations verify the analytical findings. The optimal convergence rate in the $L^2$ and the $H^1$ norm and the sub-optimal convergence rate in the FCT norm are proved.

The structure of the paper is as follows: In Sec.~\ref{sec:model_problem} we introduce the FEM-FCT scheme and give an example of the limiter that will be used in the simulations. In Sec.~\ref{sec:stability} we provide a variational formulation of the scheme and prove the stability of the linear as well as the nonlinear version of the scheme in the FCT norm. In Sec.~\ref{sec:error_estimates} we prove the finite element error analysis of the scheme using standard interpolation estimates. Next, in Sec.~\ref{sec:numerics} we provide the numerical simulations of the scheme on three different types of the grid. Lastly, in Sec.~\ref{sec:summary} we present a summary and provide an outlook. 

\section{Preliminaries}\label{sec:model_problem}
Let \(\Omega\in\mathbb{R}^d,\; d\in\{2,3\}\), be a bounded domain with boundary \(\partial\Omega\) and \([0,T]\) be a bounded time interval. Consider the evolutionary convection-diffusion-reaction equation: Find \(u:(0,T] \times \Omega\rightarrow \mathbb{R}\) such that
\begin{equation}\label{eq:time_cdr_eqn}
\begin{array}{rcll}
u'-\varepsilon \Delta u+\bb\cdot \nabla u+cu&=&f \quad &\mbox{in } (0,T]\times \Omega,\\
u &=&0 & \mbox{in } (0,T] \times \partial\Omega,\\
u(0,x) & = & u_0  & \mbox{in } \Omega.
\end{array}
\end{equation}
Here \(0<\varepsilon\ll1\) is a diffusivity parameter, \(\bb(t,x)\) is the convection field, \(c(t,x)\) is the reaction coefficient, \(f(t,x)\) is the given outer source of unknown quantity \(u\), and \(u_0(x)\) is the initial condition. The use of homogeneous Dirichlet boundary conditions in the system \eqref{eq:time_cdr_eqn} is just for  simplicity of presentation. Furthermore, it is assumed that there exists a positive constant \(c_0\) such that
\begin{equation}\label{eq:asmpt}
  c(t,x) -\frac12 \nabla \cdot \bb(t,x) \ge c_0 >0 
\end{equation}
which guarantees the unique solvability of Eq.~\eqref{eq:time_cdr_eqn} (see \cite{RST08}).

This paper's main topic is the analysis of the FEM-FCT schemes described in \cite{Ku09}. To keep the paper self-contained, a short presentation of these schemes will be given here. Consider a spatial discretization of Eq.~\eqref{eq:time_cdr_eqn} using the FEM with piecewise linear elements. For temporal discretization, a backward Euler scheme is used. These schemes, in algebraic form, lead at the discrete-time \(t_n\) to an equation of the form
\begin{equation}\label{eq:fd_high_ord}
M_Cu^n+\tau Au^n =\tau f^n  + M_Cu^{n-1},
\end{equation}
where \(\tau \) is the time step length. The matrix \(M_C\) is the consistent mass matrix, and the stiffness matrix $A$ is the sum of diffusion, convection, and reaction. Furthermore, the notations $u^n,\;f^n$ stand for the vectors unknown coefficients of the finite element method. 

In order to satisfy the maximum principle for the discrete problem, the system matrix should be an  \(M\)-matrix. The sufficient condition for a matrix to be an \(M\)-matrix is that all the diagonal entries are positive, all off-diagonal entries are non-positive, and the row sum is positive \cite[Theorem~2]{Cia70}. To achieve this, we modify the left hand side of  Eq.~\eqref{eq:fd_high_ord} such that the system matrix corresponds to an \(M\)-matrix. Let us define the artificial diffusion matrix \(D\) such that
\begin{equation}\label{eq:matD}
d_{ij}=-\mathrm{max}\{a_{ij},0,a_{ji}\} \quad \mbox{for } i\neq j,\qquad  d_{ii}=-\sum_{j=1,j\neq i}^Nd_{ij},
\end{equation}
where \(N\) is the number of degrees of freedom and lumped mass matrix \(M_L\) 
\begin{equation}\label{eq:matML}
M_L=\mathrm{diag}\left(m_i\right), \quad m_i=\sum_{j=1}^Nm_{ij}.
\end{equation}
Then, the matrix $\mathbb{A}=A+D$ and $M_L$ satisfies the condition for a \(M\)-matrix and the following scheme is a stable low-order scheme
\begin{equation}\label{eq:fd_low_order}
M_Lu^n+\tau \mathbb{A}u^n =\tau f^n +M_Lu^{n-1}+f^*(u^n,u^{n-1}).
\end{equation}
In the next step, one needs to define an appropriate ansatz for $f^*(u^n,u^{n-1})$. To this end, subtracting Eq.~\eqref{eq:fd_high_ord} from Eq.~\eqref{eq:fd_low_order} one get the residual vector
\[
r=\left( M_L-M_C\right)u^n+(\mathbb{A}-A)u^n\tau -(M_L-M_C)u^{n-1}.
\]
The idea is to limit these modifications by introducing solution dependent limiters $\alpha_{ij}$ such that
\[f^*_i(u^n,u^{n-1})=\sum_{j=1}^N\alpha_{ij}f	_{ij}\qquad i=1,\ldots,N.\]

The basic idea of FEM-FCT is to find these weights in such a way that they are close to \(1\) in smoother regions (i.e., we recover the Galerkin FEM) and close to \(0\) in the vicinity of layers (to recover the low-order scheme). The contribution of $f^*_i(u^n,u^{n-1})$ stems from a decomposition of the residual vector
\[
r_i=\sum_{j=1}^Nf_{ij}=\sum_{j=1}^Nm_{ij}(u_i^n-u_j^n)-\sum_{j=1}^Nm_{ij}(u_i^{n-1}-u_j^{n-1})+\sum_{j=1}^Nd_{ij}(u_j^n-u_i^n)\tau \]
for \(i=1,\ldots,N\). The above representation is derived from the definition of $D$ and \(M_L\). The number \(f_{ij}\) are referred as fluxes.

In order to have a conservative scheme, the limiters \(\alpha_{ij}\) have to be symmetric i.e.,
\begin{equation}\label{eq:limiter_symmetry}
\alpha_{ij}=\alpha_{ji},
\end{equation}
Now, we can rewrite Eq.~\eqref{eq:fd_low_order} in an algebraic form
\begin{eqnarray}\label{eq:fem_fct}
\lefteqn{\sum_{j=1}^N m_{ij}(u_j^n-u_j^{n-1}) +  \tau \sum_{j=1}^Na_{ij}u_j^n} \nonumber \\
&& +\sum_{j=1}^Nm_{ij}(1-\alpha_{ij})\left[(u_i^n-u_i^{n-1})-(u_j^n-u_j^{n-1})\right]\nonumber\\
&& +\tau \sum_{j=1}^Nd_{ij}(1-\alpha_{ij})(u_j^n-u_i^n)=f_i^n\tau \quad \mathrm{for }\ i=1,\ldots,N,
\end{eqnarray}
where $\alpha_{ij}\in [0,1]$, $i,j=1,\ldots,N$ satisfy Eq.~\eqref{eq:limiter_symmetry}.

The scheme Eq.~\eqref{eq:fem_fct} can be handled in two different ways. The nonlinear version of the FEM-FCT scheme utilizes an explicit solution \(\bar{u}\) with the forward Euler scheme at \(t_n-\tau /2\)
\begin{equation}\label{eq:forward_euler_approximation}
\bar{u}=u^{n-1}-\frac{\tau }{2}M_L^{-1}\left( \mathbb{A}u^{n-1}-f^{n-1}\right).
\end{equation}
Here \(\bar{u}\) will be used as a prelimiting of the fluxes for the nonlinear scheme.
If $f_{ij}(\bar{u}_i-\bar{u}_j)<0$ then set $f_{ij}=0$, which is prescribed in~\cite{KM05, Ku09}.

The linear FEM-FCT is a special case presented in~\cite{Ku09}. Consider the vector $u^n$ in the flux $f_{ij}$ which is replaced by an approximation obtained with an explicit scheme. Using $u^{n-1/2}=(u^n+u^{n-1})/2$ in the definition of $f_{ij}$ leads to
\begin{eqnarray*}
f_{ij} & = &2m_{ij}(u_i^{n-1/2}-u_i^{n-1})-2m_{ij}(u_j^{n-1/2}-u_j^{n-1})\nonumber \\
&& + 2\tau d_{ij}(u_j^{n-1/2}-u_i^{n-1/2})+\tau d_{ij}(u_i^{n-1}-u_j^{n-1}).\nonumber
\end{eqnarray*}

We can approximate $u^{n-1/2}$ by the forward Euler method in the same way as the pre-limiting in the nonlinear scheme. Inserting Eq.~\eqref{eq:forward_euler_approximation} leads to
\begin{equation*}
f_{ij}=\tau  m_{ij}(\nu_i^{n-1/2}-\nu_j^{n-1/2}) + \tau  d_{ij}\left[u_j^{n-1}-u_i^{n-1}+\tau  (\nu_j^{n-1/2}-\nu_i^{n-1/2})\right]
\end{equation*}
where $\nu_i^{n-1/2}=\left(M_L^{-1}(f_{n-1}-\mathbb{A}u^{n-1}\right)_i$.

Note that both methods use an explicit method as predictor, which results in a CFL condition for these methods, for details see \cite{KM05, Ku09}.
\subsection{Limiters}\label{sec:limiters}
In this section we give an example of limiter that will be used in the simulations. We follow the defintion of the Zalesak algorithm presented in \cite{JN12}.
\begin{enumerate}
    \item Compute 
    \begin{equation*}
       P_i^+ = \sum_{j=1,j\neq i}^N f_{ij}^+ \qquad P_i^- = \sum_{j=1,j\neq i}^N f_{ij}^-.
    \end{equation*}
    \item Compute 
    \begin{eqnarray}
       Q_i^+ & = &\max\left\{ 0,\max_{i=1,\ldots,N,j\neq i}(\bar{u}_j-\bar{u}_i) \right\},\nonumber \\
       Q_i^- & = &\min\left\{ 0,\min_{i=1,\ldots,N,j\neq i}(\bar{u}_j-\bar{u}_i) \right\}.\nonumber
    \end{eqnarray}
    \item Compute
    \begin{equation*}
       R_i^+ = \min\left\{ 1,\frac{m_iQ_i^+}{P_i^+} \right\}, \qquad R_i^+ = \min\left\{ 1,\frac{m_iQ_i^-}{P_i^-} \right\}.
    \end{equation*}
    If the \(P_i^+\) or \(P_i^-\) is zero, we set \(R_i^+=1\) or \(R_i^-=1\), respectively. 
    \item Compute
    \[ \alpha_{ij} = \begin{cases} 
        \min\{R_i^+,R_j^-\} & \mbox{ if } f_{ij}>0\\
        \min\{R_i^-,R_j^+\} & \mbox{ otherwise},
      \end{cases}
   \]
\end{enumerate}
where $f_{ij}^+=\max \lbrace f_{ij},0\rbrace$ and $f_{ij}^+=\min \lbrace f_{ij},0\rbrace$.

\section{Stability of the FEM-FCT Methods}\label{sec:stability}
This section studies the stability of the fully discrete version of the linear and nonlinear FEM-FCT schemes. To this end, we first write  Eq.~\eqref{eq:time_cdr_eqn} into a variational form, and then using the finite element discretization, we mention the variational formulation of the nonlinear algebraic problem. 

Let \(V\subset H^1_0(\Omega)\). A variational formulation of Eq.~\eqref{eq:time_cdr_eqn} reads: Find \(u:(0,T]\rightarrow V\)  with such that
\begin{equation}\label{eq:wf}
(u',v)+\varepsilon(\nabla u,\nabla v)+(\bb\cdot \nabla u+cu,v)=(f,v)\qquad \forall\ v\in V.
\end{equation}
Here, \((\cdot,\cdot)\) denotes the inner product in \(L^2(\Omega)^d\). For the finite element discretization of Eq.~\eqref{eq:wf}, \(V\) is replaced by the finite-dimensional space \(V_h\), where \(h\) represents the mesh size of the underlying triangulation \(\{\mathcal{T}_h\}\) of \(\Omega\). 
We consider in this paper the conforming finite element method and, therefore, \(V_h\subset V\). The numerical analysis of the AFC schemes has been derived for first-order elements on triangles, i.e.,  we choose $V_h=\mathbb{P}_1$.

The variational formulation for the nonlinear FEM-FCT scheme reads: Find \(u_h:(0,T]\rightarrow V_h\) with \(u_h(0) = u_{h,0}\) such that
\begin{equation}\label{eq:variational_non_linear_fem_fct}
(u'_h,v_h)+a_h(u_h,v_h)+d_h^D(u_h;u_h,v_h)+d_h^M(u_h;u_h,v_h)=(f,v_h)\qquad\forall\ v_h\in V_h,
\end{equation}
where  
\begin{eqnarray} \label{eq:nl_forms}
a_h(u_h,v_h)& = & (\varepsilon\nabla u_h,\nabla v_h)+(\bb\cdot \nabla u_h, v_h)+(cu_h,v_h),\nonumber \\
d_h^D(u_h;u_h,v_h)& = &\sum_{i,j=1}^N(1-\alpha_{ij})d_{ij}(u_{hj}-u_{hi})v_{hi},
\nonumber\\
d_h^M(u_h;u_h,v_h)& = & \sum_{i,j=1}^N(1-\alpha_{ij})m_{ij}\left( u_{hi}'- u_{hj}'\right) v_{hi},
\end{eqnarray}
and $u'_h$ represents the time derivative of $u$.

Note that \(u_{h,0}\in V_h \) is a suitable approximation of \(u_0\) in the finite dimensional space \(V_h\). 

The backward Euler scheme as the temporal discretization of Eq.~\eqref{eq:variational_non_linear_fem_fct} lead at the discrete-time \(t_n\) to an
equation of the form
\begin{eqnarray}\label{eq:fdp_nl}
\lefteqn{\left(u_h^n,v_h\right)+\tau  \left[a_h(u_h^n,v_h)+d_h^D(u_h^n;u_h^n,v_h)+d_h^M(u_h^n;u_h,v_h)\right]} \nonumber \\
&= & \tau (f^n,v_h)+\left(u_h^{n-1},v_h\right)\qquad\forall\ v_h\in V_h.
\end{eqnarray}
Similarly, the discrete version of the linear FEM-FCT scheme after the temporal discretization leads to the fully discrete problem
\begin{eqnarray}\label{eq:fdp_lp}
\lefteqn{(u_h^n,v_h)+\tau  \left[a_h(u_h^n,v_h)+\bar{d}_h^D(u_h^n,v_h)
+\bar{d}_h^M(u_h^n,v_h)\right]}\nonumber \\
& = &\tau  (f^n,v_h)+ \tau (f^*,v_h)+\left(u_h^{n-1},v_h\right)\qquad \forall\ v_h\in V_h,
\end{eqnarray}
where
\begin{eqnarray*}
\bar{d}_h^D(u_h,v_h)& = &\sum_{i,j=1}^Nd_{ij}(u_{hj}-u_{hi})v_{hi},\nonumber \\
\bar{d}_h^M(u_h,v_h)& = &\sum_{i,j=1}^Nm_{ij}\left( u_{hi}' -u_{hj}'\right) v_{hi},\nonumber \\
(f^*,v_h)& = &\sum_{i,j=1}^N\alpha_{ij}\left[ (\tau d_{ij}-m_{ij})(\nu_{hj}^{n-1/2}-\nu_{hi}^{n-1/2})\right]v_{hi}\nonumber \\
&& +\sum_{i,j=1}^Nd_{ij}\alpha_{ij}(u_{hj}^{n-1}-u_{hi}^{n-1})v_{hi}.\nonumber
\end{eqnarray*}
For the analysis, some preliminaries are introduced. Assuming that the meshes are quasi-uniform, the following inverse inequality (see \cite[Lemma~4.5.3]{BS08}) holds for each $v_h\in V_h$
\begin{equation}\label{eq:inverse_inequality}
\|v_h\|_{W^{m,q}(K)}\leq C_{\mathrm{inv}}h^{l-m-d(1/q'-1/q)}\|v_h\|_{W^{l,q'}(K)},
\end{equation}
where $0\leq l \leq m\leq 1$, $1\leq q'\leq q\leq \infty$, and $\|\cdot \|_{W^{m,q}(K)}$ is the norm in \(W^{m,q}(K)\). 
The norm and the semi-norm in \(W^{m,q}\) are given by \(\|\cdot\|_{m,q}\) and \(|\cdot|_{m,q}\), respectively. In case \(q=2\), we write \(H^m(\Omega)\), \(\|\cdot\|_{m}\), and \(|\cdot|_{m}\) instead of \(W^{m,q}(K)\), \(\|\cdot\|_{m,q}\), and \(|\cdot|_{m,q}\), respectively.

Consider $v\in H^1(K)$ and $E\subset \partial K$, then the following trace inequality holds (see \cite[Lemma~3.1]{Ver98})
\begin{equation}\label{eq:trace_inequality}
\|v\|_{0,E}\leq C_{\mathrm{T_1}}h_K^{-1/2}\|v\|_{0,K}+C_{\mathrm{T_2}}h_K^{1/2}\|\nabla v\|_{0,K},
\end{equation}
where $C_{\mathrm{T_1}}$ and $C_{\mathrm{T_2}}$ are constants independent of $h$.

Let $m_{ij}$ be an element of the mass matrix $M_C$, then the estimate
\begin{equation}\label{eq:mij}
|m_{ij}|\leq C_{m_{ij}}h^d,
\end{equation}
holds true, where $C_{m_{ij}}$ is independent of \(h\).
\begin{lemma} \emph{(\cite[Lemma~5.23]{Jha20})}\label{lem:tang_derivative}
Let E be an edge with length $h_E$ and $v$ be a linear function on E, then
\begin{equation}\label{eq:lemma_1}
\|\nabla v\cdot \bt_E\|_{0,E}^2\leq \|\nabla v\|_{0,E}^2,
\end{equation}
where $\bt_E$ is the tangent unit vector to E.
\end{lemma}

The next lemma states the coercivity of FEM-FCT scheme and the proof can be found in~\cite{BJK16}.
\begin{lemma}
Let Eq.~\eqref{eq:asmpt} be satisfied. Then the bilinear form \[a_{\mathrm{FCT}}(u,v)=a_h(u,v)+d_h(u;u,v)\] 
associated with the FEM-FCT scheme is coercive with respect to the \(\|\cdot\|_{\mathrm{FCT}}\) norm, i.e.,
\begin{align}\label{eq:coercive}
a_{\mathrm{FCT}}(u,u)\geq C_{\mathrm{FCT}}\|u\|_{\mathrm{FCT}}^2,
\end{align}
where $C_{\mathrm{FCT}}$ is the coercive constant and the \(\|\cdot\|_{\mathrm{FCT}}\) norm is defined by
\[
\|u_h\|_{\mathrm{FCT}}:=\left(\varepsilon|u_h|_{1}^2+c_0\|u_h\|_{0}^2+d_h(u_h;u_h,u_h)\right)^{1/2}.
\]
\end{lemma}

For linear elements \(u,v\in V_h\), one can represent \(d_h^D (\cdot;\cdot,\cdot)\) through an edge formulation (see \cite[Eq.~(16)]{BJKR18})
\begin{equation}\label{eq:d_h_D_edge_formulation}
d_h^D(w;u,v)=\sum_{E\in \mathcal{E}_h}(1-\alpha_E(w))|d_E|h_E(\nabla u\cdot \bt_E,\ \nabla v\cdot \bt_E)_E,
\end{equation}
where $\mathcal{E}_h$ is the set of all edges, $d_E$ denotes $d_{ij}$, and $\alpha_E$ denotes $\alpha_{ij}$ along the edge $E$ with endpoints $x_i$ and $x_j$.

In the same way, we can represent \(d_h^M (\cdot;\cdot,\cdot)\) also through an edge formulation
\begin{equation}\label{eq:d_h_edge_formulation}
d_h^M(w;u,v)=\sum_{E\in \mathcal{E}_h}(1-\alpha_E(w))|m_E|h_E(\nabla u'\cdot \bt_E,\ \nabla v\cdot \bt_E)_E,
\end{equation}
where $m_E$ denotes $m_{ij}$. Similarly we can represent \(\underline{d}_h^M (\cdot,\cdot)\) (and \(\underline{d}_h^D (\cdot,\cdot)\)) in an edge formulation.
We note that we have a predictor-corrector scheme for FEM-FCT algorithm and hence, we need to prove stability estimates for both the predictor as well as the corrector step.

The forward Euler (FE) scheme at time $t_{n+1}-(\tau)/2$ gives
$$
\bu^n=u^{n-1}-\frac{\tau}{2}M_L^{-1}(\mathbb{A}u^{n-1}-f_{n-1}).
$$
Our variational formulation for the FE scheme looks like: Find $\bu_h^n\in V_h$ such that

\begin{equation}\label{eq:var_form_fe}
\left(u^{\mathrm{FE}}, v\right)+a(u_h^{n-1},v)+\bar{d}_h^D(u_h^{n-1},v)+\check{d}_h^M(u^{\mathrm{FE}},v)=(f_{n-1},v),
\end{equation}
where
\begin{eqnarray*}
u^{\mathrm{FE}}& = & 2\left(\frac{\bu_h^n-u_h^{n-1}}{\tau}\right)\qquad \mathrm{and,}\nonumber \\
\check{d}_h^M(u^{\mathrm{FE}},v)& = &\frac{2}{\tau}\sum_{i,j=1}^Nm_{ij}\left(\bu^i_n-\bu^j_n-u^i_{n-1}+u^j_{n-1}\right)v^i.
\end{eqnarray*}
\begin{theorem}\label{thm:stab_fe_fem_fct} 
Let Eq.~\eqref{eq:asmpt} hold. With the additional condition
\begin{equation}\label{eq:assumption_space_time}
\tau \leq Ch^2,
\end{equation}
then the solution \(\bu_h^n\) of Eq.~\eqref{eq:var_form_fe} satisfies at $t_n=n\tau $
\begin{eqnarray}\label{eq:stability_FEM_FCT_fe}
\lefteqn{\|\bu_h^n\|_0^2+C\tau \sum_{m=1}^n\|u_h^m\|_{\mathrm{FCT}}^2}\nonumber \\
&\leq &\|\bu_h^0\|_0^2+\sum_{E\in \mathcal{E}_h}|m_E|h_E\|\nabla u_h^{0}\cdot \bt_E\|_{0,E}^2+ C\tau\sum_{m =1}^n\|f_m\|_0^2,
\end{eqnarray}
where \(C\) is a constant that do not depend on \(h\) and \(\tau \).
\end{theorem}
\begin{proof}
Taking $v=\bu_h^n$ in Eq.~\eqref{eq:var_form_fe} and using
$$
\left( \bu_h^n-u_h^{n-1},\bu_h^n\right)=\frac{1}{2}\left( \|\bu_h^n\|_0^2-\|u_h^{n-1}\|_0^2+\|\bu_h^n-u_h^{n-1}\|_0^2\right),
$$
adding $a(\bu_h^n,\bu_h^n)$ and $d_h(\bu_h^n,\bu_h^n)$ on both sides and using the Cauchy-Schwarz inequality we get
\begin{eqnarray}\label{eq:discrete_var_form}
\lefteqn{\frac{1}{\tau}\big{(} \|\bu_h^n\|_0^2-\|u_h^{n-1}\|_0^2+\|\bu_h^n-u_h^{n-1}\|_0^2\big{)}}\nonumber \\
&&+a(\bu_h^n,\bu_h^n)+d_h^D(\bu_h^n,\bu_h^n)+\check{d}_h^M(u^{\mathrm{FE}},\bu_h^n)\nonumber \\
&\leq &\|f_{n-1}\|_0\|\bu_h^n\|_0+|a(\bu_h^n-u_h^{n-1},u_h^n)|+|d_h^D(\bu_h^n-u_h^{n-1},\bu_h^n)|.
\end{eqnarray}
Now using the Cauchy-Schwarz inequality, Eq.~\eqref{eq:inverse_inequality}, and Young's inequality
\begin{eqnarray}\label{eq:bilinear_form_discrete}
\left|a(\bu_h^n-u_h^{n-1},\bu_h^n)\right| & = &\big|\varepsilon(\nabla (\bu_h^n-u_h^{n-1}),\nabla \bu_h^n)+(\bb\cdot \nabla (\bu_h^n-u_h^{n-1}),\bu_h^n)\nonumber \\
&&+c(\bu_h^n-u_h^{n-1},\bu_h^n)\big|\nonumber \\
&\leq &\varepsilon |\bu_h^n-u_h^{n-1}|_1|\bu_h^n|_1+\|\bb\|_{\infty}|\bu_h^n-u_h^{n-1}|_1\|\bu_h^n\|_0 \nonumber \\
&&+\|c\|_{\infty}\|\bu_h^n-u_h^{n-1}\|_0\|\bu_h^n\|_0\nonumber \\
&\leq & \left(\varepsilon^{1/2} \cin^2h^{-1}+\|\bb\|_{\infty}\cin h^{-1}+\|c\|_{\infty}\right)\nonumber \\
&& \times \|\bu_h^n-u_h^{n-1}\|_0\|\bu_h^n\|_{\mathrm{FCT}}\nonumber \\
&\leq &\left(\varepsilon^{1/2} \cin^2h^{-1}+\|\bb\|_{\infty}\cin h^{-1}+\|c\|_{\infty}\right)^2\nonumber \\
&& \times \frac{3}{2}\|\bu_h^n-u_h^{n-1}\|_0^2 +\frac{\|\bu_h^n\|_{\mathrm{FCT}}^2}{6}.
\end{eqnarray}
Next using the edge formulation Eq.~\eqref{eq:d_h_D_edge_formulation}, the Cauchy-Schwarz inequality, the trace inequality, the inverse estimate, and Young's inequality we get
\begin{eqnarray}\label{eq:diffusion_discrete}
|d_h^D(\bu_h^n-u_h^{n-1},\bu_h^n)|&\leq & \sum_{E\in \mathcal{E}_h}|d_E|h_E\|\nabla (\bu_h^n-u_h^{n-1})\cdot \bt_E\|_{0,E}\|\nabla \bu_h^n\cdot \bt_E\|_{0,E}\nonumber \\
&\leq & \sum_{E\in \mathcal{E}_h}|d_E|h_E\|\nabla (\bu_h^n-u_h^{n-1})\|_{0,E}\|\nabla \bu_h^n\|_{0,E}\nonumber \\
&\leq & \underset{E\in \mathcal{E}_h}{\mathrm{max}}(|d_E|)\left(\sum_{E\in \mathcal{E}_h}h_E\|\nabla (\bu_h^n-u_h^{n-1})\|_{0,E}^2\right)^{1/2}\nonumber \\
&& \times \left(\sum_{E\in \mathcal{E}_h}h_E\|\nabla \bu_h^{n}\|_{0,E}\right)^{1/2}\nonumber \\
&\leq &\underset{E\in \mathcal{E}_h}{\mathrm{max}}(|d_E|)\|\nabla (\bu_h^n-u_h^{n-1})\|_0\|\nabla \bu_h^n\|_0\nonumber \\
&\leq & C_{d_E}\cin h^{-1}\|\bu_h^n-u_h^{n-1}\|_0\|\nabla\bu_h^n\|_0\nonumber \\
&\leq & C_{d_E}^2\left( \frac{3h^{-2}}{2\varepsilon}\|\bu_h^n-u_h^{n-1}\|_0^2\right)+\frac{\|\bu_h^n\|_{\mathrm{FCT}}^2}{6},
\end{eqnarray}
where $C_{d_E}$ is the maximum of $|d_E|$ and independent of $h$.

Lastly, we have to approximate $\check{d}_h^M(u^{\mathrm{FE}},\bu_h^n)$. We can have an edge representation of this term similar to Eq.~\eqref{eq:d_h_edge_formulation}
$$
\check{d}_h^M(u^{\mathrm{FE}},\bu_h^n)=\sum_{E\in \mathcal{E}_h}\frac{2}{\tau}|m_E|h_E\left( \nabla (\bu_h^n-u_h^{n-1})\cdot \bt_E, \nabla \bu_h^n\cdot \bt_E\right)_E.
$$
By simple algebraic manipulation we get
\begin{eqnarray}\label{eq:time_discrete_1}
\check{d}_h^M(u^{\mathrm{FE}},\bu_h^n)& = &\frac{2}{\tau}\Bigg{(} \sum_{E\in \mathcal{E}_h}|m_E|h_E\|\nabla (\bu_h^n-u_h^{n-1})\cdot \bt_E\|_{0,E}^2\nonumber \\
&&-\sum_{E\in \mathcal{E}_h}|m_E|h_E\|\nabla u_h^{n-1}\cdot \bt_E\|_{0,E}^2\nonumber \\
&&+\sum_{E\in \mathcal{E}_h}|m_E|h_E(\nabla \bu_h^n\cdot \bt_E,\nabla u_h^{n-1}\cdot \bt_E)\Bigg{)},
\end{eqnarray}
and we also have
\begin{eqnarray}\label{eq:time_discrete_2}
\check{d}_h^M(u^{\mathrm{FE}},\bu_h^n)& = &\frac{2}{\tau}\Bigg{(} \sum_{E\in \mathcal{E}_h}|m_E|h_E\|\nabla \bu_h^n\cdot \bt_E\|_{0,E}^2\nonumber \\
&& -\sum_{E\in \mathcal{E}_h}|m_E|h_E(\nabla \bu_h^n\cdot \bt_E,\nabla u_h^{n-1}\cdot \bt_E)\Bigg{)}.
\end{eqnarray}
Adding Eq.~\eqref{eq:time_discrete_1} and Eq.~\eqref{eq:time_discrete_2} we finally have
\begin{eqnarray}\label{eq:time_discrete}
\check{d}_h^M(u^{\mathrm{FE}},\bu_h^n) & = &\frac{1}{\tau}\Bigg{(} \sum_{E\in \mathcal{E}_h}|m_E|h_E\|\nabla \bu_h^n\cdot \bt_E\|_{0,E}^2\nonumber \\
&& +\sum_{E\in \mathcal{E}_h}|m_E|h_E\|\nabla (\bu_h^n-u_h^{n-1})\cdot \bt_E\|_{0,E}^2\nonumber \\
&&-\sum_{E\in \mathcal{E}_h}|m_E|h_E\|\nabla u_h^{n-1}\cdot \bt_E\|_{0,E}^2\Bigg{)}.
\end{eqnarray}
Substituting Eq.~\eqref{eq:bilinear_form_discrete}, Eq.~\eqref{eq:diffusion_discrete}, and Eq.~\eqref{eq:time_discrete} in Eq.~\eqref{eq:discrete_var_form} and using
\begin{eqnarray*}
C_{\mathrm{FCT}}\|\bu_h^n\|_{\mathrm{FCT}}^2& \leq& a(\bu_h^n,\bu_h^n),\nonumber \\
0& \leq &d_h^D(\bu_h^n,\bu_h^n)\qquad \mathrm{and,}\nonumber \\
0& \leq &\sum_{E\in \mathcal{E}_h}|m_E|h_E\|\nabla (\bu_h^n-u_h^{n-1})\cdot \bt_E\|_{0,E}^2,\nonumber
\end{eqnarray*}
we get
\begin{eqnarray*}
\lefteqn{\|\bu_h^n\|_0^2+\|\bu_h^n-u_h^{n-1}\|_0^2+\sum_{E\in \mathcal{E}_h}|m_E|h_E\|\nabla \bu_h^n\cdot \bt_E\|_{0,E}^2+\tau C_{\mathrm{FCT}}\|u_h^n\|_{\mathrm{FCT}}^2}\nonumber \\
&\leq &C\tau\|f_{n-1}\|^2_0\nonumber \\
&& +C\frac{\tau}{h^2}\left(\varepsilon \cin^2+\|\bb\|_{\infty}\cin +\|c\|_{\infty}h + \frac{C^2_{d_E}}{\varepsilon}\right)\|\bu_h^n-u_h^{n-1}\|_0^2 \nonumber \\
&& +\sum_{E\in \mathcal{E}_h}|m_E|h_E\|\nabla u_h^{n-1}\cdot \bt_E\|_{0,E}^2+\|u_h^{n-1}\|_0^2.\nonumber
\end{eqnarray*}
Summing over $n=1,\ldots,N$ and bounding $\sum_{E\in \mathcal{E}_h}|m_E|h_E\|\nabla u_h^{N}\cdot \bt_E\|_{0,E}^2$ by below we have stability whenever $\tau\leq Ch^2$, where $C$ is independent of $h$.
\end{proof}

The next theorem provides the stability of the nonlinear FEM-FCT scheme Eq.~\eqref{eq:fdp_nl}. 
\begin{theorem}\label{thm:stab-nl_fem_fct} 
Let Eq.~\eqref{eq:asmpt} hold. With the additional condition
\begin{equation}\label{eq:assumption_limiter}
\alpha_E\leq \frac{1}{\left(C_{\mathrm{T_1}}C_{\mathrm{inv}}C_{m_E}\right)^2}\mathrm{min}\left\lbrace \tau ,\frac{c_0}{2}\right\rbrace,
\end{equation}
then the solution \(u_h^n\) of Eq.~\eqref{eq:fdp_nl} satisfies at $t_n=n\tau $
\begin{equation}\label{eq:stability_non_linear_FEM_FCT}
\|u_h^n\|_0^2+C\tau \sum_{m=1}^n\|u_h^m\|_{\mathrm{FCT}}^2\leq \|u_h^0\|_0^2+\frac{1}{C_{\mathrm{T_1}}^2}\|\nabla u_h^0\|_0^2+ \frac{\tau }{c_0}\sum_{m =1}^n\|f^m\|_0^2
\end{equation}
where \(C\) and \(C_{T_1}\) are constants that do not depend on \(h\) and \(\tau \).
\end{theorem}
\begin{proof}
The proof follows the standard approach. Setting \(v_h=u_h^m\) in Eq.~\eqref{eq:fdp_nl}, we get
\begin{eqnarray*}
\left( u_h^m -u_h^{m -1},u_h^m \right) &+&
\tau \left[a_h(u_h^m ,u_h^m ) +d_h^D(u_h^m ;u_h^m ,u_h^m )
+d_h^M(u_h^m ;u_h,u_h^m )\right] \nonumber \\
& = & \tau (f^m ,u_h^m ).\nonumber
\end{eqnarray*}
Using \((a-b,a)=\frac{1}{2}\left( \|a\|_0^2 - \|b\|_0^2 + \|a-b\|_0^2\right) \) for $a,b\in L^2(\Omega)$, Eq.~\eqref{eq:nl_forms}, and Eq.~\eqref{eq:coercive}, it follows that
\begin{eqnarray}\label{eq:temp_eqn_1}
\lefteqn{\frac{1}{2}\left(  \|u_h^m \|_0^2-\|u_h^{m -1}\|_0^2+\|u_h^m -u_h^{m-1}\|_0^2\right)}\nonumber \\
& + & C_a\tau  \|u_h^m \|_{\mathrm{FCT}}^2
+\tau \sum_{i,j=1}^Nm_{ij}\left(u_{hi}'- u_{hj}' \right)u_{hi}^m \nonumber \\
& \leq & \tau \sum_{i,j=1}^N\alpha_{ij}m_{ij}\left( u_{hi}'- u_{hj}' \right)u_{hi}^m  + \tau  (f^m,u_h^m ).
\end{eqnarray}
Using \(u'=(u^m -u^{m -1})/\tau \) and the Cauchy-Schwarz inequality for the last term on the left-hand side of Eq.~\eqref{eq:temp_eqn_1}, we get 
\begin{eqnarray*}
\lefteqn{\tau  \sum_{i,j=1}^Nm_{ij}\left(u_{hi}'- u_{hj}' \right)u_{hi}^m}\nonumber \\
& = &\sum_{E\in \mathcal{E}_h} |m_E|h_E \left(\nabla (u_h^m -u_h^{m -1})\cdot \bt_E,\nabla u_h^m \cdot \bt_E\right)_{0,E}\nonumber \\
& = &\sum_{E\in \mathcal{E}_h} |m_E|h_E \Big\lbrace\left(\nabla (u_h^m -u_h^{m -1})\cdot \bt_E,\nabla (u_h^m -u_h^{m -1})\cdot \bt_E\right)_{0,E}\nonumber \\
&& +  \left(\nabla (u_h^m -u_h^{m -1})\cdot \bt_E,\nabla u_h^{m -1}\cdot \bt_E\right)_{0,E}\Big\rbrace\nonumber \\
& = &\sum_{E\in \mathcal{E}_h} |m_E|h_E \Big\lbrace\|\nabla (u_h^m -u_h^{m -1})\|_{0,E}^2 \nonumber \\
&& + \left(\nabla u_h^m \cdot \bt_E,\nabla u_h^{m -1}\cdot \bt_E\right)_{0,E} 
	 - \|\nabla u_h^{m -1}\cdot \bt_E\|_{0,E}^2\Big\rbrace.
\end{eqnarray*}
One can also estimate the same term as follows
\begin{eqnarray*}
\lefteqn{\tau  \sum_{i,j=1}^Nm_{ij}\left(u_{hi}'- u_{hj}' \right)u_{hi}^n}\nonumber \\
&& = \sum_{E\in \mathcal{E}_h} |m_E|h_E \left\{
	\|\nabla u_h^{m }\cdot \bt_E\|_{0,E}^2 - \left(\nabla u_h^m \cdot \bt_E,\nabla u_h^{m -1}\cdot \bt_E\right)_{0,E}
	\right\}.
\end{eqnarray*}
Adding the above estimate, we get
\begin{eqnarray}\label{eq:expression_for_d_h1}
\lefteqn{	\tau  \sum_{i,j=1}^Nm_{ij}\left(u_{hi}'- u_{hj}' \right)u_{hi}^m}\nonumber \\
& = & \frac12 \sum_{E\in \mathcal{E}_h} |m_E|h_E 
	\Big\lbrace
	\|\nabla (u_h^m -u_h^{m -1})\|_{0,E}^2 + \|\nabla u_h^{m }\cdot \bt_E\|_{0,E}^2 \nonumber \\
&&- \|\nabla u_h^{m -1}\cdot \bt_E\|_{0,E}^2
	\Big\rbrace.
\end{eqnarray}
The first term on the right-hand side of Eq.~\eqref{eq:temp_eqn_1} uses the edge formulation, 
the Cauchy-Schwarz inequality, estimate Eq.~\eqref{eq:lemma_1} and the Young's inequality
\begin{eqnarray*}
\tau \sum_{i,j=1}^N\alpha_{ij}m_{ij}\left( u_{hi}'- u_{hj}' \right)u_{hi}^m  
&=& \tau \sum_{E\in \mathcal{E}_h}\alpha_E|m_E|\left( \nabla u_h'.\cdot \bt_E,
\nabla u_h^m\cdot \bt_E\right)_E  h_E\\
&\le& \tau \sum_{E\in \mathcal{E}_h}\alpha_E\;|m_E|\;\|\nabla u'_h\cdot \bt_E\|_{0,E}\nonumber \\
&&\times \|\nabla u_h^m \cdot \bt_E\|_{0,E} h_E \\
&\le& \tau \sum_{E\in \mathcal{E}_h}\alpha_E\;|m_E|\;\|\nabla u_h'\|_{0,E}\|\nabla u_h^m \|_{0,E} h_E\\
& \le &\tau \sum_{E\in \mathcal{E}_h}\frac{\alpha_E\;|m_E|\; h_E}{2} \|\nabla u_h'\|_{0,E}^2\nonumber \\
&& + \tau \sum_{E\in \mathcal{E}_h}\frac{\alpha_E\;|m_E|\;h_E}{2} \|\nabla u_h^m \|_{0,E}^2.
\end{eqnarray*} 
Now, using the bounds Eq.~\eqref{eq:mij} and Eq.~\eqref{eq:assumption_limiter}, the local trace inequality Eq.~\eqref{eq:trace_inequality} and an inverse inequality Eq.~\eqref{eq:inverse_inequality} gives
\begin{eqnarray*}
\tau \sum_{i,j=1}^N\alpha_{ij}m_{ij}\left( u_{hi}'- u_{hj}' \right)u_{hi}^m  
&=&\sum_{E\in \mathcal{E}_h}\frac{h^3\tau ^2}{2\left(C_{\mathrm{T_1}}C_{\mathrm{inv}}\right)^2}\|\nabla u_h'\|_{0,E}^2\nonumber \\
&&+\sum_{E\in \mathcal{E}_h}\frac{c_0 h^3\tau }{4\left(C_{\mathrm{T_1}}C_{\mathrm{inv}}\right)^2}\|\nabla u_h^m \|_{0,E}^2\nonumber\\
&\leq & \sum_{K\in \mathcal{T}_h}\frac{\tau ^2}{2}\|u_h'\|_{0,K}^2+\sum_{K\in \mathcal{T}_h}\frac{c_0 \tau }{4} \|u_h^n\|_{0,K}^2\nonumber\\
&=& \sum_{K\in \mathcal{T}_h}\frac{1}{2}\|u^m _h-u^{m -1}_h\|_{0,K}^2\nonumber \\
&&+\sum_{K\in \mathcal{T}_h}\frac{c_0}{4}\tau \|u_h^m \|_{0,K}^2\nonumber\\
&\leq &\frac{1}{2}\|u_h^m -u_h^{m -1}\|_0^2+\frac{c_0 \tau}{4} \|u_h^m \|_0^2.
\end{eqnarray*}
The estimate for the second term on the right-hand side of Eq.~\eqref{eq:temp_eqn_1} uses the Cauchy-Schwarz inequality, and the Young’s inequality to get
\begin{equation*}
\tau (f^m ,u_h^m ) \leq \frac{\tau}{c_0 }  \|f^m \|^2_0+\frac{c_0\tau}{4}
\|u_h^m \|_0^2.
\end{equation*}

Collecting the above estimates in Eq.~\eqref{eq:temp_eqn_1} and the fact that the first 
term in Eq.~\eqref{eq:expression_for_d_h1} is positive and can bounded by zero, we get
\begin{multline*}
\frac{1}{2}\|u_h^m \|_0+C\tau \|u_h^m \|_{\mathrm{FCT}}^2+\frac{1}{2}\sum_{E\in \mathcal{E}_h}|m_E|h_E\|\nabla u_h^m \cdot \bt_E\|_{0,E}^2\\
\leq \frac{1}{2}\|u_h^{m -1}\|_0^2+\frac{\tau}{c_0} \|f_h^m \|_0^2+\frac{1}{2}\sum_{E\in \mathcal{E}_h}|m_E|h_E\|\nabla u_h^{m -1}\cdot \bt_E\|_{0,E}^2.
\end{multline*}
Summing over $m =1,\dots,n$ leads to 
\begin{multline*}
\frac{1}{2}\|u_h^n\|_0+C\tau \sum_{m =1}^{n}\|u_h^j\|_{\mathrm{FCT}}^2+\frac{1}{2}\sum_{E\in \mathcal{E}_h}|m_E|h_E\|\nabla u_h^n\cdot \bt_E\|_{0,E}^2\\
\leq \frac{1}{2}\|u_h^0\|_0^2+\frac{\tau}{c_0} \sum_{j=1}^{n}\|f_h^j\|_0^2+\frac{1}{2}\sum_{E\in \mathcal{E}_h}|m_E|h_E\|\nabla u_h^{0}\cdot \bt_E\|_{0,E}^2,
\end{multline*}
and then bounding the terms
\begin{eqnarray*}
\frac{1}{2}\sum_{E\in \mathcal{E}_h}|m_E|\;h_E\;\|\nabla u_h^n\cdot \bt_E\|_{0,E}^2 & \ge & 0\quad \mathrm{and,}\nonumber \\
\frac{1}{2}\sum_{E\in \mathcal{E}_h}|m_E|\;h_E\;\|\nabla u_h^{0}\cdot \bt_E\|_{0,E}^2 & \le & \frac{1}{C_{\mathrm{T_1}}^2}\|\nabla u_h^0\|_0^2
\end{eqnarray*}
gives the statement of the theorem.
\end{proof}
\begin{remark}
	The assumption Eq.~\eqref{eq:assumption_limiter} on the limiter is only needed in the region where the convection is dominant. For the regions away from layers, the value of \(\alpha_E\) is close to \(1\) and the effect of stabilization vanishes.
\end{remark}

The stability of the linear FEM-FCT scheme is given in the next theorem.
\begin{theorem}\label{thm:stab-lin_fem_fct}
Let Eq.~\eqref{eq:asmpt} and the conditions of Theorem~\ref{thm:stab_fe_fem_fct} be fulfilled. Then, the solution of Eq.~\eqref{eq:fdp_lp} satisfies 	at \(t_n=n\tau\)
\begin{eqnarray}
\lefteqn{\|u_h^n\|_0^2+C\tau \sum_{m =1}^n\|u_h^j\|_{\mathrm{FCT}}^2}\nonumber \\
& \leq &\|u_h^0\|_0^2+\frac{2\tau}{c_0}\sum_{m =1}^n\left(\|f^j\|_0^2
+\|f^{*(j-1)}\|_0^2\right)+\frac{1}{C^2_{\mathrm{T}_1}}\|\nabla u_h^0\|_0^2.
\end{eqnarray}
\end{theorem}
\begin{proof}
The proof starts with the same lines of arguments as Theorem~\ref{thm:stab-nl_fem_fct}. 
We have for \(v_h=u_h^m \)
\begin{eqnarray}\label{eq:lin_stab_t1}
\lefteqn{\frac{1}{2}\left( \|u_h^m \|_0^2-\|u_h^{m -1}\|_0^2+\|u_h^m -u_h^{m-1}\|_0^2\right)
+ C_a\tau  \|u_h^m \|_{\mathrm{FCT}}^2}\nonumber \\
& + &\tau  \sum_{i,j=1}^N m_{ij}\left( u_{hi}' -u_{hj}'\right) u_{hi}^m
\leq \tau (f^m ,u_h^m )+\tau (f^{*(m -1)},u_h^m ).
\end{eqnarray}
Comparing this with Eq.~\eqref{eq:temp_eqn_1}, it can be seen that the difference is only the right hand side. 

Applying the Cauchy-Schwarz inequality followed by Young's inequality for the first term on the right hand side gives
\begin{align*}
\tau  (f^m ,u_h^m )&\leq \frac{\tau }{c_0}\|f^m \|_0^2+\frac{c_0\tau}{4}\|u_h^m \|_0^2,
\end{align*}

Inserting this estimate and Eq.~\eqref{eq:expression_for_d_h1} in Eq.~\eqref{eq:lin_stab_t1} and ignoring the terms with positive contribution to get
\begin{eqnarray*}
\lefteqn{\frac12 \|u_h^m \|_0^2+ C_a\tau  \|u_h^m \|_{\mathrm{FCT}}^2
+ \frac12\sum_{E\in \mathcal{E}_h} |m_E|h_E \|\nabla u_h^{m }\cdot \bt_E\|_{0,E}^2} \nonumber \\
& \leq & \frac12 \|u_h^{m -1}\|_0^2 + \frac12\sum_{E\in \mathcal{E}_h} |m_E|h_E \|\nabla u_h^{m -1}\cdot \bt_E\|_{0,E}^2\nonumber \\
&&+\frac{\tau }{c_0}\|f^m \|_0^2+
\tau\left(f^{*(m -1)},u_h^m\right).
\end{eqnarray*}
To approximate the last term we will use the stability of the predictor step.
\begin{eqnarray*}
(f^*(m-1),u_h^m) & = &\sum_{i,j=1}^N\alpha_{ij}f_{ij}u_h^{mi}\nonumber \\
& = &\sum_{i,j=1}^N\alpha_{ij}\Big{(}2m_{ij}\left[ \bu_h^{mi}-u_h^{(m-1)i}-\bu_h^{mj}+u_h^{(n-1)j}\right]u_h^{mi}\nonumber \\
&&+\tau d_{ij}\left( 2\bu_h^{jm}-2\bu_h^{im}+u_h^{(m-1)i}-u_h^{(m-1)j}\right)u_h^{mi}\Big{)}\nonumber \\
& = &\sum_{i,j=1}^N\alpha_{ij}\Big{(}2(-m_{ij}+\tau d_{ij})\left(\bu_h^{mj}-\bu^{mi}_h\right)u_h^{mi}\nonumber \\
&& +(-2m_{ij}+\tau d_{ij})\left(u_h^{(n-1)i}-u_h^{(n-1)j}\right)u_h^{mi}\Big{)}.\nonumber
\end{eqnarray*}

Adding and subtracting $(-2m_{ij}+\tau d_{ij})\left(u_h^{ni}-u_h^{mj}\right)u_h^{mi}$ on the right-hand side, we get
\begin{eqnarray*}
(f^*,u_h^{m}) & = &\sum_{i,j=1}^N\alpha_{ij}f_{ij}u_h^{mi}\nonumber \\
& = & \sum_{i,j=1}^N\alpha_{ij}\Big{(}2(-m_{ij}+\tau d_{ij})\left(\bu_h^{mj}-\bu^{mi}_h\right)u_h^{mi}\nonumber \\
&& +(-2m_{ij}+\tau d_{ij})\left(u_h^{mj}-u_h^{mi}+u_h^{(m-1)i}-u_h^{(m-1)j}\right)u_h^{mi}\Big{)}\nonumber \\
&& +(-2m_{ij}+\tau d_{ij})\left(u_h^{mi}-u_h^{mj}\right)u_h^{mi}\Big{)}.\nonumber
\end{eqnarray*}

The last term in the last equation can be taken to the left side. We note that all the terms can be written in edge formulation. Bounding $(-2m_{ij}+\tau d_{ij})\left(u_h^{mi}-u_h^{mj}\right)u_h^{mi}$ by below using \cite[Lemma~1]{BJK16}. To bound the rest of terms we use the Cauchy-Schwarz inequality, Eq.~\eqref{eq:inverse_inequality}, and Young's inequality,
\begin{eqnarray}\label{eq:f_star_approx}
\tau(f^{*(m-1)},u^n)&\leq & 2\tau\sum_{E\in \mathcal{E}_h}\alpha_E|m_E+\tau d_E|h_E\|\nabla \bu^m_h\|_{0,E}\|\nabla u^m\|_{0,E}\nonumber \\
&& +\tau\sum_{E\in \mathcal{E}_h}\alpha_E|2m_E+\tau d_E|h_E\nonumber \\
&&\times \|\nabla (u^{(m-1)}_h-u^m_h)\|_{0,E}\|\nabla u^m_h\|_{0,E}\nonumber \\
&\leq & 2\tau\ \underset{E\in \mathcal{E}_h}{\mathrm{max}}(|m_E+\tau d_E|)\left(\sum_{E\in\mathcal{E}_h}h_E\|\nabla \bu^m_h\|_{0,E}^2\right)^{1/2}\nonumber \\
&& \times \left(\sum_{E\in \mathcal{E}_h}h_E\|\nabla u^m_h\|_{0,E}^2\right)^{1/2}+\tau\ \underset{E\in \mathcal{E}_h}{\mathrm{max}}(|2m_E+\tau d_E|)\nonumber \\
&& \times \left(\sum_{E\in \mathcal{E}_h}h_E\|\nabla (u^{m-1}_h-u^m_h)\|_{0,E}^2\right)^{1/2}\nonumber \\
&& \times \left(\sum_{E\in \mathcal{E}_h}h_E\|\nabla u^m\|_{0,E}^2\right)^{1/2}\nonumber \\
&\leq & \tau\ C\|\nabla \bu^m_h\|_0\|\nabla u^m_h\|_0+\tau\ C\|\nabla (u^{m-1}_h-u^m_h)\|_0\|\nabla u^m_h\|_0\nonumber \\
&\leq &\tau\ CC_{\mathrm{inv}}h^{-1}\|\bu^n\|_0\|\nabla u^n\|_0\nonumber \\
&&+\tau\ CC_{\mathrm{inv}}h^{-1}\|\nabla (u^{m-1}_h-u^m_h)\|_0\|\nabla u^m_h\|_0\nonumber \\
&\leq &\frac{4\tau C}{\varepsilon\ h^2}\|\bu^m_h\|_0^2+\frac{4\tau C}{\varepsilon\ h^2}\|\nabla (u^{m-1}_h-u^m_h)\|_0^2+\frac{\tau}{8}\|u^m_h\|_a^2.
\end{eqnarray}

In Eq.~\eqref{eq:f_star_approx} the last term can be taken to the left side and the second term again is bounded by the inequality $\tau\leq Ch^2$ from  the predictor step.

Summing over the time steps \(m =1,\ldots,n\) gives
\begin{multline*}
\|u_h^n\|_0^2+ C_a\tau  \sum_{m =1}^n\|u_h^j\|_{\mathrm{FCT}}^2
+ \sum_{E\in \mathcal{E}_h} |m_E|h_E \|\nabla u_h^{n}\cdot \bt_E\|_{0,E}^2 \\
\leq \|u_h^0\|_0^2 + \sum_{E\in \mathcal{E}_h} |m_E|h_E \|\nabla u_h^0\cdot \bt_E\|_{0,E}^2+\frac{2\tau }{c_0} \sum_{m =1}^n\left(\|f^j\|_0^2+C\|\bu^m_h\|_0^2\right).
\end{multline*}
Finally, ignoring the terms with positive contribution, bounding the last term from above by the predictor step stability condition, and the application of the local trace inequality gives the statement of the theorem. 	
\end{proof}

\section{Finite Element Error Analysis}\label{sec:error_estimates}
This section details the error analysis of the linear and nonlinear FEM-FCT schemes. Sufficient regualrity condition for the solution $u$ of Eq.~\eqref{eq:time_cdr_eqn} is assumed for a  priori analysis. The analysis of both schemes starts by decomposing the error into an interpolation error and the difference of interpolation and the solution i.e.,
\[
u_h^n-u(t_n)=\left(u_h^n-\Pi_hu(t_n)\right)+\left(\Pi_hu(t_n)-u(t_n)\right),
\]
where $\Pi_h$ is a stable interpolation operator.

The following error estimates are used to get the estimates for the terms involving interpolation error 
\begin{equation}\label{eq:interpolation_estimate}
\|u-\Pi_hu\|_0+h\|u-\Pi_hu\|_1\leq Ch^{2}\|u\|_2
\end{equation}
for \(u\in V\cap H^2(\Omega)\) (see \cite{BS08}). 

In the analysis below, the following semi-norm property will be used
\begin{align}\label{eq:semi-norm}
|d_h^D(w;z,v)|^2 \le d_h^D(w;z,z)d_h^D(w;v,v),
\end{align}
which is valid for $d_h^M(\cdot; \cdot, \cdot)$ as well see \cite[Eq.~(41)]{BJK16}.

In order to get the bounds for the term \(d_h^D(\cdot;\cdot,\cdot)\), we will use the following lemma, for proof we refer to  \cite[Lemma 16]{BJK16}.
\begin{lemma}\label{lem:estim-dhd}
Let the matrix \(D\) be defined by Eq.~\eqref{eq:matD}. Then, there exists a constant \(C\) that does not depend on \(\tau\), \(h\) 
and the data of Eq.~\eqref{eq:time_cdr_eqn} such that
\begin{eqnarray}\label{eq:estim-dhd}
\lefteqn{d_h^D(w_h,\Pi_h u, \Pi_h u)}\nonumber \\
& \le & C (\varepsilon + \|\bb\|_{0,\infty,\Omega} h+\|c\|_{0,\infty,\Omega} h^2) |\Pi_h u|_1^2 \qquad \forall w_h \in V_h, \; u\in C(\overline{\Omega}).
\end{eqnarray}
\end{lemma}
This lemma also holds true for time dependent $\bb$ and $c$, as we assume them to be bounded in time as well. If we follow the same procedure as in the proof of \cite[Lemma 16]{BJK16}, we can get the following estimate.
\begin{lemma}\label{lem:estim-dhm}
Let the matrix \(M_L\) be defined by Eq.~\eqref{eq:matML} and let Eq.~\eqref{eq:mij} hold. Then, there exists a constant \(C\) that does not 
depend on \(\tau\), \(h\) and the data of Eq.~\eqref{eq:time_cdr_eqn} such that
\begin{align}\label{eq:estim-dhm}
d_h^M(w_h,\Pi_h u, \Pi_h v) \le C h |\Pi_h u_t|_1^2 + C  \|\Pi_h v\|_0^2\qquad \forall w_h \in V_h, \; u,v\in C(\overline{\Omega}).
\end{align}
\end{lemma}
For simplicity of presentation, let us denote \[\Pi_h^n u=\Pi_hu(t_n) \qquad \mbox{and} \qquad e_h^n=u_h^n-\Pi_h^n u.\] 
The a priori error analysis below assumes
\begin{align}\label{eq:regularity}
u,u_t\in L^{\infty} (0,T;H^2),\ u_{tt}\in L^2(0,T;H^1)
\end{align}
for the solution \(u\) of the Eq.~\eqref{eq:time_cdr_eqn}.
\begin{theorem}\label{thm:nl_err_est}
	Let \(\bb\in L^{\infty}(0,T;(L^{\infty})^d),\ \nabla\cdot \bb,\;c \in L^{\infty}(0,T;L^{\infty})\) for the coefficients in
	Eq.~\eqref{eq:wf}. Further, assume that the solution \(u\) satisfies the regularity assumption~Eq.~\eqref{eq:regularity}. Then, the error 
	\(u_h^n-u(t_n)\) satisfies
	\begin{eqnarray}
\lefteqn{\|u_h^n-u(t_n)\|_0^2+\tau \sum_{m=1}^n\|u_h^m -u(t_m )\|_{\mathrm{FCT}}^2} \nonumber \\
	&\leq &C\Big[ h^4+\tau^2 + (1+\varepsilon) h^2
	 +\varepsilon+\|\bb\|_{\infty}h		+\|c\|_{\infty}h^2 \Big]
	\end{eqnarray}
	where \(C\) is a constant that depends on \(u,\;u'\), and \(u''\).
\end{theorem}
\begin{proof}
	From the stability result of Theorem~\ref{thm:stab-nl_fem_fct}, we have for \(v_h=e_h^m \)
	\begin{eqnarray*}
	\lefteqn{\left(e_h^m -e_h^{m -1},e_h^m \right)
	+\tau\left[ a_h(e_h^m ,e_h^m )+d_h^D(u_h;e_h^m ,e_h^m )+d_h^m (u_h;e_h,e_h^m )\right]}\nonumber \\
	 &= &\tau (f^m ,e_h^m )-\left(\Pi_h^m u-\Pi_h^{m-1}u,e_h^m \right)\nonumber \\
	 &&-	 \tau \left[ a_h(\Pi_h^m u,e_h^m ) +d_h^D(u_h^m ;\Pi_h^m u,e_h^m )+d_h^m (u_h^m ;\Pi_hu,e_h^m )
	 \right].
	\end{eqnarray*}
	Applying the stability techniques to the above estimate, one gets
	\begin{eqnarray}\label{eq:err_nl_main}
		\|e_h^n\|_0^2+C\tau \sum_{m =1}^n\|e_h^m \|_{\mathrm{FCT}}^2& \leq &\|e_h^0\|_0^2
		+\frac{1}{C_{\mathrm{T}_1}^2}\|\nabla e_h^0\|_0^2
		+\sum_{m =1}^n(N_1^m ,e_h^m ) \nonumber \\
		&& +\sum_{m =1}^n(N_2^m ,e_h^m )-\tau\sum_{m =1}^n\Big(d_h^D(u_h^m ;\Pi_h^m u,e_h^m )\nonumber \\
        && +d_h^M(u_h^m ;\Pi_hu,e_h^m )\Big),
	\end{eqnarray}
	where \(N_1^m \) and \(N_2^m \) are given by
	\begin{eqnarray*}
		\left(N_1^m ,e_h^m \right) & = & \tau \left(u'(t_m )-\frac{(\Pi_h^m u-\Pi_h^{m -1}u)}{\tau},e_h^m \right)\nonumber \\
		&& +\left(\bb\cdot \nabla \left(u(t_m )-\Pi_h^m u\right),e_h^m \right)
		+\left(c\left(u(t_m )-\Pi_h^m u\right),e_h^m \right), \nonumber \\
		\left(N_2^m ,e_h^m \right)
		& = &\tau \varepsilon\left( \nabla \left(u(t_m )-\Pi_h^m u\right),
		\nabla e_h^m \right).
	\end{eqnarray*}
The estimates for the term \(N_1\) can also be found in \cite{JN11}. For completeness, we detail the estimates here. The application of the Cauchy-Schwarz inequality, Young's inequality, 
the triangular inequality and the interpolation error estimates 
Eq.~\eqref{eq:interpolation_estimate} gives
\begin{eqnarray*}
|(N_1^m ,e_h^m )| &\le & \frac{\tau}{2c_0} \|N_1^m \|_0^2 + \frac{\tau c_0}{2} \|e_h^m \|_0^2 \nonumber\\
&\le &C h^2 \left[\|u'(t_m )\|_{2}^2+\|c\|^2_{\infty}\|u(t_m )\|_{2}^2
+\|\bb\|_{\infty}^2\|u(t_m )\|_{2}^2\right]\nonumber \\
&& +C \tau \left\|\Pi_h^m  u(t_m ) - \frac{\Pi_h^m u - \Pi_h^{m -1} u}{\tau} \right\|_0^2
+ \frac{\tau c_0}{2} \|e_h^m \|_0^2.
\end{eqnarray*}
The term with the backward difference can be estimated using the Taylor formula with integral remainder form, the property that the time derivative and the interpolation 
\(\Pi_h\) commute, the Cauchy-Schwarz ineqaulity and the stability of \(\Pi_h\), one gets
\begin{eqnarray*}
\lefteqn{\left\|\Pi_h^m  u(t_m ) - \frac{\Pi_h^m u - \Pi_h^{m -1} u}{\tau} \right\|_0^2 }\nonumber \\
&\le & \frac{1}{\tau^2} \left\| \int_{t_{m -1}}^{t_m } (t-t_{m -1}) \Pi_h^m  u'' \right\|_0^2 \\
&\le &\frac{1}{\tau^2} \Bigg( \left(\int_{t_{m -1}}^{t_m } (t-t_{m -1})^2 dt \right)^{1/2}\left( \int_{t_{m -1}}^{t_m } \|\Pi_h^m  u''\|_0^2\right)\Bigg) \\
&\le &C \tau \int_{t_{m -1}}^{t_m } \left\|\Pi_h^m  u''\right\|_1^2 = C \tau \left\| \Pi_h^m  u''\right\|_{L^2(t_{m -1}, t_m ; H^1)}^2.
\end{eqnarray*}
It follows that
\begin{eqnarray*}
\lefteqn{|(N_1^m ,e_h^m )| }\nonumber \\
&\le & C h^2 \left[\|u'(t_m )\|_{2}^2+\|c\|^2_{\infty}\|u(t_m )\|_{2}^2
+\|\bb\|_{\infty}^2\|u(t_m )\|_{2}^2\right]\nonumber \\
&& +C \tau^2 \left\| \Pi_h^m  u''\right\|_{L^2(t_{m -1}, t_m ; H^1)}^2
+ \frac{\tau c_0}{2} \|e_h^m \|_0^2.\nonumber
\end{eqnarray*}
For \(N_2^m \), the application of Cauchy-Schwarz inequality, Young's inequality and Eq.~\eqref{eq:interpolation_estimate} gives
\begin{align*}
\left(N_2^m ,e_h^m \right) \le \frac{\varepsilon \tau h^2 }{2} \|u(t_m )\|_2^2 + \frac{\varepsilon \tau  }{2}\;|e_h^m |_1^2.
\end{align*}
For the fifth term on the right-hand side of Eq.~\eqref{eq:err_nl_main}, the seminorm property Eq.~\eqref{eq:semi-norm}, Young's inequality
and Lemma~\ref{lem:estim-dhd} give
\begin{align*}
\tau |d_h^D(u_h^m ;\Pi_h^m  u, e_h^m ) |& \le \frac\tau2 d_h^D(u_h^m ;\Pi_h^m  u, \Pi_h^m u) +\frac\tau2 d_h^D(u_h^m ;e_h^m  , e_h^m ) \\
&\le C\tau \left(\varepsilon + \|\bb\|_\infty h + \|c\|_\infty h^2 \right) |\Pi_h^m  u|_1^2 +\frac\tau2 d_h^D(u_h^m ;e_h^m  , e_h^m ) .
\end{align*}
Similarly, for the sixth term on the right-hand side of Eq.~\eqref{eq:err_nl_main} using Lemma~\ref{lem:estim-dhm}, we obtain 
\begin{align*}
\tau |d_h^M(u_h^m ;\Pi_h^m  u, e_h^m ) | \le C\tau h^2 |\Pi_hu'(t_m )|_1^2 + \frac{C\tau}{2}   \|e_h^m \|_0^2.
\end{align*}
Collecting all the estimates in Eq.~\eqref{eq:err_nl_main}, absorbing the similar terms to the left-hand side, and using \(e_h^0 = 0\), we get
\begin{eqnarray*}
\lefteqn{\|e_h^n\|_0^2+C\tau \sum_{m =1}^n\|e_h^m \|_{\mathrm{FCT}}^2}\nonumber \\ 
&\le & \sum_{m=1}^n\Big[ Ch^2 \Big[  \|u'(t_m )\|_2^2 
  + \left( \|\bb\|_{0,\infty}^2+\|c\|_{0,\infty}^2\right)\|u(t_m )\|_2^2
\Big]\nonumber \\
&& + C \tau^2 \|u''\|_{L^2(t_{m -1},t_m ;H^1)} + C \tau h^2 |u'(t_m )|_1^2\Big].
\end{eqnarray*}
The statement of the theorem then follows by applying the triangular inequality and the interpolation error estimates.
\end{proof}
In the next theorem, an error estimate for the linear FEM-FCT scheme is derived.
\begin{theorem}\label{thm:existence_linear}
Let \(\bb \in L^{\infty}(0,T;(L^{\infty})^d),\ \nabla\cdot \bb,\;c	\in L^{\infty}(0,T;L^{\infty})\) for the coefficients in
	Eq.~\eqref{eq:wf}. Further, assume that the solution \(u\) satisfies the regularity 
	assumption~Eq.~\eqref{eq:regularity}. Then, the error \(u_h^n-u(t_n)\) satisfies
	\begin{eqnarray}
\lefteqn{\|u_h^n-u(t_n)\|_0^2+\tau \sum_{m =1}^n\|u_h^m -u(t_m )\|_{\mathrm{FCT}}^2}\nonumber \\
	&\leq & C\Big[ h^4+\tau^2 +(1+\varepsilon) h^2
	 +(1+\tau)\left(\varepsilon+\|\bb\|_{\infty}h	+\|c\|_{\infty}h^2\right) \Big]
	\end{eqnarray}
	where \(C\) is a constant that depends on the \(u,\;u'\), and \(u''\).
\end{theorem}
\begin{proof}
The error analysis for the error in the linear FEM-FCM starts by taking \(v_h=e_h^m \) in the stability 
Theorem~\ref{thm:stab-nl_fem_fct}, using \(e_h^0=0\),
and applying similar estimates to get
\begin{eqnarray}\label{eq:lin_er_es_t1}
\lefteqn{\|e_h^n\|_0^2+\tau C\sum_{m =1}^n\|e_h^m \|_{\mathrm{FCT}}^2}\nonumber\\
&\leq&  \tau \sum_{m =1}^n(f^{*(m -1)},e_h^n)+\sum_{m =1}^n\left(N_1^m +N_2^m ,e_h^m \right)\nonumber \\
&&-\tau \sum_{m =1}^nd_h^D(\Pi_h^m u,e_h^m )-\tau \sum_{m =1}^Nd_h^M(\Pi_hu,e_h^m ).
\end{eqnarray}
The second term on the right-hand side is estimated in Theorem~\ref{thm:nl_err_est}. Also, the bounds for the last two terms can be derived by using the same arguments as in Theorem~\ref{thm:nl_err_est}
\begin{align*}
\tau d_h^D(\Pi_h^m u,e_h^m ) &\le 
C\tau \left(\varepsilon +\|\bb\|_{\infty}h+\|c\|_{\infty}h^2 \right)
\;\left|\Pi_h^m  u\right|_{1}^2
+\frac\tau2 d_h^D(e_h^m ,e_h^m )\\
\tau d_h^M(\Pi_h u,e_h^m ) &\leq C\tau  h^2 \left|\Pi_h u_t(t_m )\right|^2_1
+\tilde{C}\tau \|e_h^m \|_0^2.
\end{align*}
For the first term on the right-hand side of Eq.~\eqref{eq:lin_er_es_t1} we note that it can be written in an edge formulation similar to Eq.~\eqref{eq:d_h_edge_formulation}. So we get,
\begin{eqnarray*}
\lefteqn{\left(f^{*(m -1)},e_h^m \right)}\nonumber \\
& = &\sum_{i,j=1}^N\alpha_{ij}\left[ (\tau  d_{ij}-m_{ij})(\Pi_h^{m -1/2}
\nu_j-\Pi_h^{m-1/2}\nu_i)\right] e_{hi}^m \nonumber \\
&& +\sum_{i,j=1}^N\alpha_{ij}d_{ij}(\Pi_h^{m -1}u_{hj}-\Pi_h^{m -1}u_{hi})e_{hi}^m \nonumber \\
& = &\sum_{E\in \mathcal{E}_h}\alpha_E h_E|\tau d_E-m_E|\left(\nabla \Pi_h^{m -1/2}\nu\cdot 
t_E,\nabla e_h^m \cdot t_E\right)\nonumber \\
&& +\sum_{E\in \mathcal{E}_h}\alpha_Eh_E|d_E|\left(\nabla \Pi_h^{m -1}u_h\cdot t_E,\nabla e_h^m \cdot t_E\right).
\end{eqnarray*}
Using the triangle inequality, $\alpha_E\leq 1$, Cauchy-Schwarz inequality, Young's inequality, and lemma \ref{lem:estim-dhd}, it follows 
that
\begin{eqnarray*}
\lefteqn{\left(f^{*(m -1)},e_h^m \right)}\nonumber \\
&\leq & \sum_{E\in\mathcal{E}_h}(\tau |d_E|+|m_E|)h_E|\Pi_h^{m -1/2}\nu\cdot t_E|_{1,E}
|e_h^m \cdot t_E|_{1,E}\nonumber \\
&& +C(\varepsilon+\|\bb\|_{\infty}h
+\|c\|_{\infty}h^2)|\Pi_h^{m -1}u_h|_{1}^2 +\frac{1}{4}\underline{d}_h^D(e_h^m ,e_h^m ), \nonumber \\
&\leq &\sum_{E\in \mathcal{E}_h}\tau ^2 |d_E|h_E|\Pi_h^{m -1/2}\nu\cdot t_E|_{1,E}^2+
\sum_{E\in \mathcal{E}_h}\frac 14|d_E|h_E|e_h^m \cdot t_E|_{1,E}^2\nonumber \\
&& +C\sum_{E\in \mathcal{E}_h}|m_E|h_E|\Pi_h^{m -1/2}\nu\cdot t_E|_{1,E}^2 +\frac14\sum_{E\in \mathcal{E}_h}|m_E|h_E|e_h^m \cdot t_E|_{1,E}^2\nonumber \\
&& +C(\varepsilon+\|\bb\|_{\infty}h+\|c\|_{\infty}h^2)|\Pi_h^{m -1}
u_h|_{1}^2+\frac{1}{4}\underline{d}_h^D(e_h^m ,e_h^m )
\end{eqnarray*}
Note that the second and the last term are the same, so adding them and using lemma \ref{lem:estim-dhd} for the first term and lemma \ref{lem:estim-dhm} for the third and fourth term, we finally get
\begin{align*}
\left(f^{*(m -1)},e_h^m \right)&\leq C\left( \tau ^2(\varepsilon+\|\bb\|_{\infty}h+\|c\|_{\infty}h^2)
+h^2\right)|\Pi_h^{m -1/2}\nu|_1^2
+\frac{C}{4}\|e_h^m \|_0^2
\\&\qquad+C(\varepsilon+\|\bb\|_{\infty}h+\|c\|_{\infty}h^2)|
\Pi_h^{m -1}u_h|_{1}^2
+\frac{1}{2}\underline{d}_h^D(e_h^m ,e_h^m ).
\end{align*}
Collecting all the estimates in Eq.~\eqref{eq:lin_er_es_t1}, absorbing the similar terms to the left-hand side, using $e_h^0=0$, the statement of the theorem follows by applying the triangular inequality and the interpolation error estimates.
\end{proof}

\begin{remark}
We observe that both in Theorem~\ref{thm:nl_err_est} and Theorem~\ref{thm:existence_linear}, we have $\mathcal{O}(\tau)$ convergence for the time-discretization and $\mathcal{O}(h^{0.5})$ in the space discretization for the FCT norm in the convection-dominated regime. It has been noted in \cite{BJK16} that for shock capturing methods such as the FEM-FCT schemes $\mathcal{O}(h^{0.5})$ convergence is expected.
\end{remark}

\section{Numerical Simulation}\label{sec:numerics}
\input{numerics}
  
\section{Summary}\label{sec:summary}
In this paper, we presented the first finite element error analysis for the flux corrected transport techniques introduced in \cite{Ku09} in the natural norm of the system referred as the FCT norm.  The Galerkin FEM was used for the space discretization coupled with the backward Euler time discretization. Numerical studies are presented in two dimensions for three different type of grids. The main finding of the analysis and the numerical simulations are given below:

\begin{enumerate}
\item We proved conditional stability of the FEM-FCT algorithm with the restriction in time step coming from the predictor step of forward Euler (see Eq.~\eqref{eq:assumption_space_time}) for both the linear and the nonlinear schemes.
\item For the analysis of the linear FEM-FCT algorithm, there was no restriction on the choice of the limiter, and only general assumptions, i.e., $\alpha_{ij}\in [0,1]$ and Eq.~\eqref{eq:limiter_symmetry} are considered.
\item For the nonlinear FEM-FCT algorithm an additional assumption on the limiter is required, i.e., Eq.~\eqref{eq:assumption_limiter}. This assumption is strong but is only required in the convection-dominated regime.
\item From the analysis, one expects $\mathcal{O}(\tau)$ convergence in the FCT norm for time-dominated problems and $\mathcal{O}(h^{0.5})$ convergence for the convection-dominated problems.
\item We obtained the predicted order of convergence in the FCT norm for constant limiter on a particular grid. For the other grids optimal order was obtained for the $L^2$, $H^1$, and the FCT norm. For the FCT norm the experimental order of convergence is better than the obtained order of convergence for grid~1 and grid~2. This is expected as we assumed the general properties of the limiters for the analysis.
\end{enumerate}
In the future, it would be interesting to use a general theta scheme instead of the backward Euler discretization, namely the Crank-Nicolson scheme. For our analysis, we assumed the limiter's general properties and would like to investigate a deeper analysis for the limiter's specific choices, e.g., \cite{Zal79, Ku20}. Further a deeper analysis of the constants appearing in the stability estimate to obtain sharper bounds would be beneficial. Lastly, the efficient solution of the nonlinear problem for the FEM-FCT algorithm remains an open question, and one should investigate sophisticated solvers and parallel techniques for obtaining the solution. For more open questions in stabilized techniques for Eq.~\eqref{eq:time_cdr_eqn}, we refer to \cite{JKN18}.
\bibliographystyle{alpha}
\bibliography{error_fem_fct}
\end{document}